  \documentclass[reqno,10pt]{amsart}
  \usepackage{latexsym}
  \usepackage[all]{xy}
  \usepackage{amsfonts}
  \usepackage{amsthm}
  \usepackage{amsmath}
  \usepackage{amssymb}
  \usepackage{pifont}
  \usepackage{a4wide}
  \usepackage[T1]{fontenc}
\normalfont

  \def\<{{\langle}}
  \def\>{{\rangle}}

  \def\note#1{{}}

  \def\note#1{}

  \def\beq{\begin{equation}}
  \def\eeq{\end{equation}}

     \def\1{\mathbf{1}}

  \newcounter{zlist}

  \newcounter{blist}

  \newcounter{rlist}

   \newcounter{alist}

\def\stac#1{\raise-.2cm\hbox{$\stackrel{\displaystyle\otimes}{\scriptscriptstyle{#1}}$}}

\def\cten#1{\raise-.2cm\hbox{$\stackrel{\displaystyle\widehat{\otimes}}
{\scriptscriptstyle{#1}}$}}

  \def\Label#1{\label{#1}\ifmmode\llap{[#1] }\else
  \marginpar{\smash{\hbox{\tiny [#1]}}}\fi}
  \def\Label{\label}

  \newtheorem{proposition}{Proposition}
  \newtheorem{lemma}[proposition]{Lemma}
  \newtheorem{corollary}[proposition]{Corollary}
  \newtheorem{theorem}[proposition]{Theorem}

  \theoremstyle{definition}

  \newtheorem{example}[proposition]{Example}

  \theoremstyle{remark}

  \theoremstyle{definition}
  \swapnumbers

\begin{document}

 \title{Finitistic Weak Dimension of Commutative Arithmetical Rings}
 \author{Fran\c{c}ois Couchot}
 \address{ Laboratoire de Math\'ematiques Nicolas Oresme, CNRS UMR
  6139,
D\'epartement de math\'ematiques et m\'ecanique,
14032 Caen cedex, France}
  \email{francois.couchot@unicaen.fr}

\vspace*{1cm}
  \begin{abstract}
It is proven that each commutative arithmetical ring $R$ has a finitistic weak dimension $\leq 2$. More precisely, this dimension is $0$ if $R$ is locally IF, $1$ if $R$ is locally semicoherent and not IF, and $2$ in the other cases.
  \end{abstract}
  \maketitle

All rings in this paper are unitary and commutative.
A  ring $R$ is said to be a {\it chain ring} if its lattice of ideals is totally ordered by inclusion, and $R$ is called {\it arithmetical} if $R_P$ is a  chain ring for each maximal ideal $P$.
If $M$ is an $R$-module, we denote by $\mathrm{w.d.}(M)$ its {\it weak dimension}. Recall that $\mathrm{w.d.}(M)\leq n$ if $\mathrm{Tor}^R_{n+1}(M,N)=0$ for each $R$-module $N$. For any ring $R$, its {\it global weak dimension} $\mathrm{w.gl.d}(R)$ is the supremum of $\mathrm{w.d.}(M)$ where $M$ ranges over all (finitely presented cyclic) $R$-modules. Its {\it finitistic weak dimension}
$\mathrm{f.w.d.}(R)$ is the supremum of $\mathrm{w.d.}(M)$ where $M$ ranges over all  $R$-modules of finite weak dimension. 

When $R$ is an arithmetical ring, by \cite{Oso69}, a paper by B. Osofsky, we know that $\mathrm{w.gl.d}(R)\leq 1$ if $R$ is reduced, and $\mathrm{w.gl.d}(R)=\infty$ otherwise. By using the small finitistic dimension, similar results are proved by S. Glaz and S. Bazzoni when $R$ is a  Gaussian ring satisfying one of the following two conditions:
\begin{itemize}
\item $R$ is locally coherent (\cite[Theorem 3.3]{Gla05});
\item $R$ contains a prime ideal $L$ such that $LR_L$ is nonzero and T-nilpotent (a slight generalization of \cite[Theorem 6.4]{BaGl07} using  \cite[Theorems P and 6.3]{Bas60}).
\end{itemize}
They conjectured that these conditions can be removed. Recall that each arithmetical ring is Gaussian.

In this paper, we only investigate the finitistic weak dimension of arithmetical rings (it seems that it is more difficult for 	Gaussian rings). Main results are  summarized in the following theorem:
\begin{theorem}
\label{T:main} Let $R$ be an arithmetical ring. Then:
\begin{enumerate}
\item $\mathrm{f.w.d.}(R)=0$ if $R$ is locally IF;
\item $\mathrm{f.w.d.}(R)=1$ if $R$ is locally semicoherent and not locally IF;
\item $\mathrm{f.w.d.}(R)=2$ if $R$ is not locally semicoherent. 
\end{enumerate}
\end{theorem} 
Let $\mathcal{P}$ be a ring property. We say that a ring $R$ is {\it locally $\mathcal{P}$} if $R_P$ satisfies $\mathcal{P}$ for each maximal ideal $P$. As  in \cite{Mat85}, a ring $R$ is said to be \textit{semicoherent} if $\mathrm{Hom}_R(E,F)$ is a submodule of a flat $R$-module for any pair of injective $R$-modules $E,\ F$. A ring $R$ is said to be \textit{IF} ({\it semi-regular} in \cite{Mat85}) if each injective $R$-module is flat. If $R$ is a chain ring, we denote by $P$ its maximal ideal, $Z$ its  subset of zerodivisors which is a prime ideal and $Q(=R_Z)$ its fraction ring. If $x$ is an element of a module $M$ over a ring $R$, we denote by $(0:x)$ the annihilator ideal of $x$ and by $\mathrm{E}(M)$ the injective hull of $M$.

\bigskip

Since flatness is a local module property,  Theorem~\ref{T:main} is an immediate consequence of 
 the following theorem that we will prove in the sequel.
\begin{theorem}\label{T:mainlocal}
Let $R$ be a chain ring. Then:
\begin{enumerate}
\item $\mathrm{f.w.d.}(R)=0$ if $R$ is IF;
\item $\mathrm{f.w.d.}(R)=1$ if $R$ is semicoherent and not IF;
\item  $\mathrm{f.w.d.}(R)=2$ if $R$ is not semicoherent. Moreover, an $R$-module $M$ has  finite weak dimension if and only if $Z\otimes_RM$ is flat.
\end{enumerate}
\end{theorem}

An exact sequence of $R$-modules $0 \rightarrow F \rightarrow E \rightarrow G \rightarrow 0$  is {\it pure}
if it remains exact when tensoring it with any $R$-module. Then, we say that $F$ is a \textit{pure} submodule of $E$. When $E$ is flat, it is well known that $G$ is flat if and only if $F$ is a pure submodule of $E$. An $R$-module $E$ is {\it FP-injective}  if 
$\hbox{Ext}_R^1 (F, E) = 0$  for any finitely presented $R$-module $F$.  We recall that a module $E$ is FP-injective if and only if it is a pure submodule
of every overmodule. 

\bigskip

By \cite[Proposition 3.3]{Mat85} the following proposition holds:
\begin{proposition}\label{P:IF1}
A ring $R$ is IF if and only if it is coherent and self FP-injective.
\end{proposition}

The following proposition will be used frequently in the sequel. 

\begin{proposition}
\label{P:ValSemiCoh} Let $R$ be a chain ring.  The following conditions are equivalent:
\begin{enumerate}
\item $R$ is semicoherent; 
\item $Q$ is an IF-ring;
\item $Q$ is coherent;
\item either $Z=0$ or $Z$ is not flat.
\item for each nonzero element $a$ of $Z$ $\mathrm{E}(Q/Qa)$ is flat.
\end{enumerate} 
\end{proposition}
\begin{proof}
By \cite[Th\'eor\`eme 2.8]{Cou82} $Q$ is self FP-injective. So, $(2)\Leftrightarrow (3)$ by Proposition~\ref{P:IF1} and $(1)\Leftrightarrow (2)$ by \cite[Corollary 5.2]{Couc09}. By applying \cite[Theorem 10]{Couch03} to $Q$ we get that $(2)\Leftrightarrow (4)\Leftrightarrow (5)$.
\end{proof}

\bigskip

The first assertion of Theorem~\ref{T:mainlocal} is an immediate consequence of the following proposition:
\begin{proposition}\label{P:IF}
For each  IF-ring $R$, $\mathrm{f.w.d.}(R)=0$.
\end{proposition}
\begin{proof}
Let $0\rightarrow F_1\rightarrow F_0\rightarrow K\rightarrow 0$ be an exact sequence where $F_1$ and $F_0$ are flat. Since $R$ is IF, $F_1$ is FP-injective (see \cite[Lemma 4.1]{Ste70}). So, $F_1$ is a pure submodule of $F_0$. Consequently $K$ is flat.  We deduce that each $R$-module $M$ satisfying $\mathrm{w.d.}(M)<\infty$  is flat.
\end{proof}

\begin{lemma}
\label{L:FlatAnn} Let $R$ be a chain ring. Then, for each nonzero element $a$ of $P$, $(0:a)$ is a module over $Q$ and it is a flat $R/(a)$-module.
\end{lemma}
\begin{proof}
Let $x\in (0:a)$ and $s$ a regular element of $R$. Since $R$ is a chain ring $x=sy$ for some $y\in R$, and $0=ax=say$.  It follows that $ay=0$. Hence the multiplication by $s$ in $(0:a)$ is bijective. 

If $c\in R$ we denote by $\bar{c}$ the coset $c+(a)$. Any element of $(\bar{c})\otimes_{R/(a)} (0:a)$ is of the form $\bar{c}\otimes x$ where $x\in (0:a)$. Suppose that $c\notin (a)$ and $cx=0$. There exists $t\in R$ such that $a=ct\ne 0$. Since $R$ is a chain ring, from $cx=0$ and $ct\ne 0$ we deduce that $x=ty$ for some $y\in R$. We have $ay=cty=cx=0$. So, $y\in (0:a)$ and $\bar{c}\otimes x=\bar{a}\otimes y=0$. Hence $(0:a)$ is flat over $R/(a)$.
\end{proof}

\begin{lemma}\label{L:R=Q}
Let $p$ be an integer $\geq 1$, $R$ a chain ring and $M$ an $R$-module. Then $\mathrm{w.d.}(M)\leq p$ if and only if $\mathrm{w.d.}(M_Z)\leq p$.
\end{lemma}
\begin{proof} Because of the flatness of $Q$, $\mathrm{w.d.}(M)\leq p$ implies $\mathrm{w.d.}(M_Z)\leq p$.

Conversely, let  $0\ne a\in R$. By Lemma~\ref{L:FlatAnn} $(0:a)$ is a module over $Q$. So, from the exact sequences 
\[0\rightarrow (0:a)\rightarrow R\xrightarrow{a} R\rightarrow R/(a)\rightarrow 0,\]
\[0\rightarrow (0:a)\rightarrow Q\xrightarrow{a} Q\rightarrow Q/Qa \rightarrow 0,\]
we deduce   for each  integer $q\geq 1$ the following isomorphisms:
\begin{equation}\label{eq:tor}
\mathrm{Tor}^R_{q+2}(R/(a),M)\cong \mathrm{Tor}^R_q((0:a),M)\cong \mathrm{Tor}^Q_q((0:a),M_Z)\cong \mathrm{Tor}^Q_{q+2}(Q/Qa,M_Z). 
\end{equation}

So, if $p>1$ we get that $\mathrm{w.d.}(M_Z)\leq p$ implies $\mathrm{w.d.}(M)\leq p$. Now assume that $p=1$. Then, for each $a\in R$ we have:
\[\mathrm{Tor}^Q_1(Qa,M_Z)\cong \mathrm{Tor}^Q_2(Q/Qa,M_Z)=0.\]
In the following commutative diagram
\[\begin{array}{ccccccc}
0 &\rightarrow & \mathrm{Tor}^R_1((a),M) & \rightarrow & (0:a)\otimes_RM & \rightarrow & M \\
& & & & \downarrow & & \downarrow \\
& & 0 & \rightarrow & (0:a)\otimes_QM_Z & \rightarrow & M_Z
\end{array}\]
the left vertical map is an isomorphism because $(0:a)$ is a $Q$-module (Lemma~\ref{L:FlatAnn}). It follows that the homomorphism $(0:a)\otimes_RM  \rightarrow  M$ is injective. We successively deduce that $\mathrm{Tor}^R_1((a),M)=0$ and $\mathrm{Tor}^R_2(R/(a),M)=0$. Hence $\mathrm{w.d.}(M)\leq 1$.
\end{proof}

\bigskip

From Lemma~\ref{L:R=Q} we deduce the second assertion of Theorem~\ref{T:mainlocal}:

\begin{proof}[Proof of Theorem~\ref{T:mainlocal}{\rm (2)}]
Assume that $R$ is semicoherent and not IF and let $M$ be an $R$-module with $\mathrm{w.d.}(M)<\infty$. Then we have $\mathrm{w.d.}(M_Z)<\infty$. Since $Q$ is IF, $M_Z$ is flat by Proposition~\ref{P:IF}. By Lemma~\ref{L:R=Q} $\mathrm{w.d.}(M)\leq 1$. Moreover, since $R\ne Q$ then $\mathrm{w.d.}(R/(a))=1$ for each $a\in P\setminus Z$. Consequently,  $\mathrm{f.w.d.}(R)=1$.
\end{proof}

\bigskip

\begin{lemma}
\label{L:Pflat} Let $R$ be a chain ring which is not semicoherent. Then $Z\otimes_RM$ is flat for each $R$-module $M$ satisfying $\mathrm{w.d.}(M)\leq 2$.
\end{lemma}
\begin{proof}
Consider the following flat resolution of $M$:
\[0\rightarrow F_2\rightarrow F_1\rightarrow F_0\rightarrow M\rightarrow 0.\]
If $a\in Z$ we have $\mathrm{Tor}^R_p((0:a),M)\cong\mathrm{Tor}^R_{p+2}(R/(a),M)=0$ for each integer $p\geq 1$. So, the following sequence is exact:
\[0\rightarrow (0:a)\otimes_RF_2\rightarrow (0:a)\otimes_RF_1\rightarrow (0:a)\otimes_RF_0\rightarrow (0:a)\otimes_RM\rightarrow 0.\]
By Lemma~\ref{L:FlatAnn} $(0:a)\otimes_RF_i$ is flat over $R/(a)$ for $i=0,1,2$. By \cite[Theorem 11(1)]{Couch03} $R/(a)$ is IF. We deduce that $(0:a)\otimes_RM$ is flat over $R/(a)$ by Proposition~\ref{P:IF}. If $K$ is the kernel of the map $F_0\rightarrow M$, it follows that the sequence $\mathcal{S}_a$ 
\[0\rightarrow  (0:a)\otimes_RK\rightarrow (0:a)\otimes_RF_0\rightarrow (0:a)\otimes_RM\rightarrow 0\]
is pure-exact over $R/(a)$, and over $R$ too.  If $0\ne r\in Z$,  there exists $0\ne a\in Z$, such that $ra=0$. Hence $Z=\cup_{a\in Z\setminus\{0\}}(0:a)$. So, if $\mathcal{S}$ is the sequence 
\[0\rightarrow  Z\otimes_RK\rightarrow Z\otimes_RF_0\rightarrow Z\otimes_RM\rightarrow 0,\]
then $\mathcal{S}=\varinjlim_{a\in Z\setminus\{0\}}\mathcal{S}_a$. By \cite[Theorem I.8.13(a)]{FuSa01} $\mathcal{S}$ is pure-exact. By Proposition~\ref{P:ValSemiCoh} $Z$ is flat. Consequently $Z\otimes_RF_0$ is flat and so is $Z\otimes_RM$.
\end{proof}

\bigskip
Now, we can prove the third assertion of Theorem~\ref{T:mainlocal}.

\begin{proof}[Proof of Theorem~\ref{T:mainlocal}{\rm (3)}]
Assume that $R$ is not semicoherent. By \cite[Proposition 14]{Couch03},  the exact sequence 
\[0\rightarrow Z\rightarrow Q\rightarrow Q/Z\rightarrow 0\]
induces the following one for each nonzero element $a$ of $Z$:
\[0\rightarrow Q/Z\rightarrow \mathrm{E}(Q/Z)\rightarrow\mathrm{E}(Q/aQ)\rightarrow 0.\] From these two exact sequences we deduce the following one:
\[0\rightarrow Z\rightarrow Q\rightarrow \mathrm{E}(Q/Z)\rightarrow\mathrm{E}(Q/aQ)\rightarrow 0.\]
It is well known that $Q$ is flat. By \cite[Proposition 8]{Couch03} $\mathrm{E}(Q/Z)$ is flat. Since $Q$ is not IF, $Z$ is flat and $\mathrm{E}(Q/aQ)$ is not flat by Proposition~\ref{P:ValSemiCoh}. So, $\mathrm{f.w.d.}(R)\geq 2$. From the isomorphism (\ref{eq:tor}) in the proof of Lemma~\ref{L:R=Q} we deduce that $\mathrm{f.w.d.}(R)=\mathrm{f.w.d.}(Q)$. So, we may assume that $R=Q$, hence the maximal ideal $P$ of $R$ is $Z$. By way of contradiction suppose that there exists   an $R$-module $M$ with $\mathrm{w.d.}(M)\geq 3$. We may replace $M$ with a syzygy module of a flat resolution and  assume that $\mathrm{w.d.}(M)=3$. We consider the following exact sequence where $F$ is flat:
\[0\rightarrow K\rightarrow F\rightarrow M\rightarrow 0.\]
Then $\mathrm{w.d.}(K)\leq 2$. By Lemma~\ref{L:Pflat} $P\otimes_RK$ is flat. So, $\mathrm{w.d.}(P\otimes_RM)\leq 1$ (recall that $P$ is flat by Proposition~\ref{P:ValSemiCoh}). Since $\mathrm{w.d.}(R/P)=1$, then $\mathrm{w.d.}(M/PM)\leq 1$ and $\mathrm{w.d.}(\mathrm{Tor}^R_1(R/P,M))\leq 1$ because these last modules are $R/P$-modules. We deduce from the exact sequence
\begin{equation}
\label{eq:one}
0\rightarrow\mathrm{Tor}^R_1(R/P,M)\rightarrow P\otimes_RM\rightarrow PM\rightarrow 0 
\end{equation}
that $\mathrm{w.d.}(PM)\leq 2$ and from the following one
\begin{equation}
\label{eq:two}
0\rightarrow PM\rightarrow M\rightarrow M/PM\rightarrow 0 
\end{equation}
that $\mathrm{w.d.}(M)\leq 2$. We get a  contradiction. So, we conclude that $\mathrm{f.w.d.}(R)=2$.

By Lemma~\ref{L:Pflat} it remains to prove that $\mathrm{w.d.}(M)\leq 2$ if $Z\otimes_RM$ is flat. As above we may assume that $R=Q$, and by using again the exact sequences~(\ref{eq:one}) and (\ref{eq:two}) we successively show that $\mathrm{w.d.}(PM)\leq 2$ and $\mathrm{w.d.}(M)\leq 2$.
\end{proof}

\bigskip

If $A$ is a proper ideal of a chain ring $R$, let $A^{\sharp}=\{r\in R\mid A\subset (A:r)\}$. Then $A^{\sharp}$ is a prime ideal containing $A$ ($A^{\sharp}/A$ is the set of zerodivisors of $R/A$).

From Theorem~\ref{T:mainlocal} and \cite[Corollary 5.3]{Couc09} we deduce the following:
\begin{corollary}
Let $A$ be a non-zero proper ideal of a chain ring $R$. The following conditions are equivalent:
\begin{enumerate}
\item $R/A$ is semicoherent;
\item $A$ is either prime or the inverse image of a non-zero proper principal ideal of $R_{A^{\sharp}}$ by the natural map $R\rightarrow R_{A^{\sharp}}$;
\item $\mathrm{f.w.d.}(R/A)\leq 1.$
\end{enumerate}
\end{corollary}

By using \cite[Theorem 11]{Couch03} we deduce the following:
\begin{corollary}
Let $A$ be a non-zero proper ideal of a chain ring $R$. The following conditions are equivalent:
\begin{enumerate}
\item $R/A$ is IF;
\item either $A=P$  or $A$ is finitely generated;
\item $\mathrm{f.w.d.}(R/A)= 0.$
\end{enumerate}
\end{corollary}

From these corollaries we deduce the following examples:
\begin{example}
\textnormal{Let $R$ be a valuation domain which is not a field. Let $a$ be a non-zero element of its maximal ideal $P$. Then $R/aR$ is IF, so $\mathrm{f.w.d.}(R/aR)=0$. Assume that $R$ contains a non-idempotent prime ideal $L$. If $a$ is a non-zero element of $L$, then $aL$ is a non-finitely generated ideal of $R_L$. So, $R/aL$ is not semicoherent and $\mathrm{f.w.d.}(R/aL)=2$. It is well known that $\mathrm{f.w.d.}(R)=\mathrm{w.gl.d}(R)=1$. Moreover, assume that $R$ contains a non-zero prime ideal $L\ne P$. Then, if $0\ne a\in L$ then $aR_L$ is an ideal of $R$ and $R/aR_L$ is semicoherent and not IF. So, $\mathrm{f.w.d.}(R/aR_L)=1$.}
\end{example}

Recall that a chain ring $R$ is {\it strongly discrete} if there is no non-zero idempotent prime ideal.

From Theorem~\ref{T:main} and \cite[Corollary 5.4]{Couc09} we deduce the following:
\begin{corollary}
Let $R$ be an arithmetical ring. The following conditions are equivalent:
\begin{enumerate}
\item $R$ is locally strongly discrete;
\item for each proper ideal $A$, $R/A$ is locally semicoherent;
\item for each proper ideal $A$, $\mathrm{f.w.d.}(R/A)\leq 1.$
\end{enumerate}
\end{corollary}

 \end{document}